\documentclass[10pt]{amsart}

\usepackage{amsmath}
\newtheorem{theorem}{Theorem}[section]
\newtheorem{lemma}[theorem]{Lemma}

\newtheorem{corollary}{Corollary}[section]
\newtheorem{proposition}{Proposition}[section]

\newtheorem{problem}{Problem}

\theoremstyle{definition}

\newtheorem{example}[theorem]{Example}

\usepackage[active]{srcltx}

\theoremstyle{remark}
\newtheorem{remark}[theorem]{Remark}

\numberwithin{equation}{section}

\begin{document}

\title{Integral curvature and topological obstructions for submanifolds}

\author{Theodoros Vlachos}
\address{Department of Mathematics, University of Ioannina, 45110 Ioannina, Greece}
\email{tvlachos@uoi.gr}

%    General info
\subjclass[2000]{Primary 53C40, 53C20; Secondary 53C42}

\keywords{Curvature tensor, second fundamental form,  scalar curvature,  $L^{n/2}$-norm of curvature, Betti numbers, flat bilinear forms}

\begin{abstract}
We provide integral curvature bounds for compact Riemannian manifolds that allow isometric immersions into a Euclidean space with low codimension
 in terms of the Betti numbers.
\end{abstract}

\maketitle

\section{Introduction}
According to Nash's embedding theorem, every Riemannian manifold can be isometrically immersed into a Euclidean space with sufficiently high codimension.
On the other hand, 
there are results that  put several restrictions on existence of isometric immmersions of a   Riemannian manifold
 in a Euclidean space, when the  codimension is low. 
A classical theorem  due to Chern, Kuiper and Otsuki \cite{CK,O} states that a compact $n$-dimensional
 Riemannian manifold of non-positive sectional curvature cannot be isometrically  immersed into $\mathbb{R}^{2n-1}$.
  Moore \cite{Mo1,Mo2,Mo3,Mo4} investigated topological restrictions on positively curved manifolds that allow isometric immersions 
 in a Euclidean space with low codimension. 
In \cite{Mo1} he proved that any compact $n$-dimensional Riemannian manifold with positive sectional curvature that 
allows an isometric immersion in $\mathbb{R}^{n+2}$ is homeomorphic to the sphere $S^{n}$. Moreover, in \cite{Mo3, Mo4} it was proved
  that a compact $n$-dimensional Riemannian manifold  with constant sectional curvature that admits an isometric immersion in
 $\mathbb{R}^{2n-1}$ is isometric to a round sphere. 

All these results  require conditions on the range of the sectional curvature. It is natural to ask whether
 the above restrictions on isometric immersions are maintained when one only assumes  integral curvature bounds. In this paper, we 
address the following more general problem:

\begin{problem}
What kind of restrictions, 
integral curvature bounds   impose on isometric immersions in the Euclidean space with low codimension? 
\end{problem}

We are interested in bounds for the
 $L^{n/2}$-norm of 
the tensors $R-kR_{1}$ and $R-\frac{{\mathrm{scal}}}{n(n-1)}R_{1}$ for compact $n$-dimensional 
 Riemannian manifolds $(M^n,g)$.  Here $R$ and ${\mbox{scal}}$ denote respectively the  curvature tensor and the scalar 
curvature of $g$,   $k$ is a constant, and $R_1=-\frac{1}{2}g\bullet g$, where $\bullet$ stands for the Kulkarni-Nomizu product. 
Particular case of the above question is the following problem that was posed by 
Shiohama and Xu \cite{SX}.

\begin{problem}
 Let $M^n$ be a  compact $n$-dimensional  Riemannian manifold ($ n\geq3$) 
 which can be isometrically immersed into $\mathbb{R}^{2n-1}$. Does there exist 
a positive constant $\varepsilon(n)$ depending only on $n$ such that if 
$$ \int _{M}\big\Vert R-\frac{{\mathrm{scal}}}{n(n-1)}R_{1}\big\Vert^{n/2}dM < \varepsilon(n),$$
 then $M^n$ is homeomorphic to $S^n$?
\end{problem}

Shiohama and Xu \cite{SX}, affirmatively answered this problem for the case of hypersurfaces.

The aim of the paper is to make a contribution to the above problems for compact
 Riemannian manifolds that allow isometric 
immersions in a Euclidean space, with low codimension $p \geq2$, by providing  topological bounds for the $L^{n/2}$-norm of 
the tensors $R-kR_{1}$ and $R-\frac{{\mathrm{scal}}}{n(n-1)}R_{1}$ .  Throughout the paper, 
all manifolds under consideration are assumed to be connected, without boundary and oriented. We prove the following results.

\begin{theorem}
Given $n \geq 4,   k \in \mathbb{R}$  and $ \lambda>0$, there exists a constant  
$\varepsilon=\varepsilon(n,k,\lambda)>0$, such that if $(M^n,g)$ is a compact, $n$-dimensional Riemannian manifold
 that admits an isometric immersion in $\mathbb{R}^{n+p}$, $2 \leq p\leq n/2 $, so that the scalar curvature and 
 the second fundamental form $\alpha$ satisfy $|\mathrm{scal}|\geq\lambda\Vert\alpha\Vert^2$, then
\begin{equation*}
 \int _{M}\big\Vert R-kR_{1}\big\Vert^{n/2}dM \geq \varepsilon(n,k,\lambda)I,
\end{equation*}
where
$$
I:=\left\{\begin{array}{l}\sum_{i=p}^{n-p}\beta_{i}\;\;\;\;\;\mbox{if}\;\;k>0,
\vspace*{1.5ex}\\
\sum_{i=0}^{n}\beta_{i}\;\;\;\;\;\; \mbox{if}\;\;k\leq 0,
\end{array} \right.
$$
and $\beta_{i}$ is the $i$-th Betti number of $M^n$ with respect to an arbitrary coefficient field.
Furthermore, 

(i)
If $k>0$ and 
\begin{equation*}
 \int _{M}\big\Vert R-kR_{1}\big\Vert^{n/2}dM < \varepsilon(n,k,\lambda),
\end{equation*}
then $M^n$ has the homotopy type of a CW-complex with no cells of dimension $i$ for $p\leq i\leq n-p$.

(ii) For each $k \leq 0, \lambda>0$ and $ a>0$, the class $ \mathcal{M}(n,k,\lambda,a)$ 
of all compact $n$-dimensional Riemannian manifolds with $\int _{M}\Vert R-kR_{1}\Vert^{n/2}dM<a$, that are isometrically
 immersed into $\mathbb{R}^{n+p}$, $2 \leq p\leq n/2 $,  
so that the scalar curvature and  the second fundamental form satisfy  $|\mathrm{scal}|\geq\lambda\Vert\alpha\Vert^2$, 
contains at most finitely many homeomorphism types.
\end{theorem}

\begin{theorem}
Given $n \geq 4$   and $\lambda>0$, there exists a constant 
$\varepsilon=\varepsilon(n,\lambda)>0$, such that if $(M^n,g)$ is a compact, $n$-dimensional Riemannian manifold
 that admits an isometric immersion in $\mathbb{R}^{n+p}$, $2 \leq p\leq n/2 $, so that the scalar curvature and 
 the second fundamental form satisfy $|\mathrm{scal}|\geq\lambda\Vert\alpha\Vert^2$, then
\begin{equation*}
 \int _{M}\Big\Vert R-\frac{{\mathrm{scal}}}{n(n-1)}R_{1}\Big\Vert^{n/2}dM
 \geq \varepsilon(n,\lambda )\sum_{i=p}^{n-p}\beta_{i}.
\end{equation*}
 In particular, if 
\begin{equation*}
 \int _{M}\Big\Vert R-\frac{{\mathrm{scal}}}{n(n-1)}R_{1}\Big\Vert^{n/2}dM < \varepsilon(n,\lambda ),
\end{equation*}
then $M^n$ has the homotopy type of a CW-complex with no cells of dimension $i$ for $p\leq i\leq n-p$. Moreover,

(i) If the scalar curvature of $M^n$ is everywhere non-positive, then
\begin{equation*}
 \int _{M}\Big\Vert R-\frac{{\mathrm{scal}}}{n(n-1)}R_{1}\Big\Vert^{n/2}dM
 \geq \varepsilon(n,\lambda )\sum_{i=0}^{n}\beta_{i}.
\end{equation*}

(ii) For each $\lambda >0$ and $a>0$, the class $ \mathcal{M}(n,\lambda,a)$ of all compact $n$-dimensional Riemannian
 manifolds with $\int _{M}\Vert R-\frac{{\mathrm{scal}}}{n(n-1)}R_{1}\Vert^{n/2}dM<a$, that are isometrically
immersed into $\mathbb{R}^{n+p}$, $2 \leq p\leq n/2 $, so that the scalar curvature 
 and the second fundamental form satisfy  $\mathrm{scal}\leq-\lambda\Vert\alpha\Vert^2$, 
contains at most finitely many homeomorphism types.
\end{theorem}

The idea for the proofs is to relate the $L^{n/2}$-norm of the tensors $R-kR_{1}$ and $R-\frac{{\mathrm{scal}}}{n(n-1)}R_{1}$
 with the Betti numbers using well-known results of Chern and Lashof \cite{CL1,CL2}, Morse theory and the Gauss equation. To 
this purpose, we need two algebraic inequalities for symmetric bilinear forms (see Propositions 2.3 and 2.4 below). 
The reason for imposing the additional assumption
on the bound of the ratio of the scalar curvature to the squared  length of the second fundamental form is 
 that these inequalities fail without this condition. 
Actually, we provide counterexamples that justify the necessity of this assumption. 

The main ingredient for the proof of the inequalities is  
 the theory of flat bilinear forms which was introduced by Moore \cite{Mo2, Mo3} as an outgrowth of Cartan's 
theory of exteriorly orthogonal quadratic forms \cite{Ca}. This theory plays an essential
 role in the study of isometric immersions (cf. \cite{CaDa,Da, Da1, Da2}).

Our results can be formulated in terms of the curvature operator $\mathcal{R}:\wedge^{2}TM\longrightarrow \wedge^{2}TM$. 
In fact, it is easy to see that 
$$ \Vert R-kR_{1}\Vert=\Vert \mathcal{R}-kId \Vert\;\; \text{and}\;\; 
\Big\Vert R-\frac{{\mathrm{scal}}}{n(n-1)}R_{1}\Big\Vert = \Big\Vert \mathcal{R}-\frac{{\mathrm{scal}}}{n(n-1)}Id \Big\Vert,$$
where $Id$  stands for the identity map on $\wedge^{2}TM$. In particular, we have the following

\begin{corollary}
 Let $(M^n,g)$ be a compact $n$-dimensional Riemannian manifold with positive curvature operator $\mathcal{R}$  
that admits an isometric immersion in $\mathbb{R}^{n+p}, 2 \leq p\leq n/2$,  so that the scalar curvature and 
 the second fundamental form satisfy $|\mathrm{scal}|\geq\lambda\Vert\alpha\Vert^2, \lambda >0$. If 
$$\int _{M}\Big\Vert \mathcal{R}-\frac{{\mathrm{scal}}}{n(n-1)}Id  \Big\Vert^{n/2}dM< \varepsilon(n,\lambda), $$
 then $M^n$ is diffeomorphic to the sphere $S^{n}$. 
\end{corollary}

\section{Algebraic auxiliary results}
This section is devoted to some algebraic results that are crucial for the proofs. Let $V, W$ be real vector spaces,
 equipped with nondegenerate inner products which by abuse of notation are both  denoted by $\langle .,. \rangle$. 
The inner product of $V$ is assumed to be positive definite. We consider the space ${\mathrm {Hom}}(V\times V, W)$ of symmetric $W$-valued
 bilinear forms on $V$. This space may be viewed as a complete metric space by means of the usual Euclidean norm $\Vert.\Vert$. 
Recall that the Kulkarni-Nomizu product $\varphi \bullet \psi$ of two bilinear forms $\varphi, \psi \in {\mathrm {Hom}}(V\times V, \mathbb{R})$
 is the (0,4)-tensor $\varphi \bullet \psi: V\times V\times V\times V\longrightarrow \mathbb{R}$ given by 
\begin{eqnarray*}
 \varphi \bullet \psi (x_{1},x_{2},x_{3},x_{4}):=\varphi(x_{1},x_{3})\psi(x_{2},x_{4}) +
 \varphi(x_{2},x_{4})\psi(x_{1},x_{3})\\ -\varphi(x_{1},x_{4})\psi(x_{2},x_{3}) - \varphi(x_{2},x_{3})\psi(x_{1},x_{4}),
\end{eqnarray*}
where ${\ }x_{1},x_{2},x_{3},x_{4} \in V$.

Using the inner product of $W$, we extend the Kulkarni-Nomizu product for bilinear forms $\beta, \gamma \in {\mathrm {Hom}}(V\times V, W)$
 as the (0,4)-tensor $\beta \bullet \gamma: V\times V\times V\times V\longrightarrow \mathbb{R}$ given by 
\begin{eqnarray*}
 \beta \bullet \gamma (x_{1},x_{2},x_{3},x_{4}):=\langle\beta(x_{1},x_{3}),\gamma(x_{2},x_{4})\rangle +
 \langle\beta(x_{2},x_{4}),\gamma(x_{1},x_{3})\rangle\\ -\langle\beta(x_{1},x_{4}),\gamma(x_{2},x_{3})\rangle -
 \langle\beta(x_{2},x_{3}),\gamma(x_{1},x_{4})\rangle.
\end{eqnarray*}
A bilinear form $\beta \in {\mathrm {Hom}}(V\times V, W)$ is called \textit{flat} with respect to the inner product
 $\langle .,.\rangle$ of $W$ if and only if 
\begin{equation*}
\langle\beta(x_{1},x_{3}),\beta(x_{2},x_{4})\rangle-\langle\beta(x_{1},x_{4}),\beta(x_{2},x_{3})\rangle=0
 \end{equation*}
for all ${\ }x_{1},x_{2},x_{3},x_{4} \in V$, or equivalently if and only if $\beta \bullet \beta=0$.

Associated to a bilinear form $\beta \in {\mathrm {Hom}}(V\times V, W)$ is the  \textit{nullity space}  $N(\beta)$  defined by 
\begin{eqnarray*}
N(\beta)=\lbrace x \in V: \beta(x,y)=0 {\ }{\ }{\mbox {for all}}{\ }{\ } y \in V \rbrace.
\end{eqnarray*} 

We need the following results on flat bilinear forms, the proofs of which can be found in \cite{Mo2,Mo3,Da}.

\begin{proposition}
 Let $\beta \in {\mathrm {Hom}}(V\times V, W)$ be a flat bilinear form with respect to a positive definite inner product of $W$. Then 
\begin{equation*}
 {\mathrm {dim}}N(\beta)\geq {\mathrm {dim}}V-{\mathrm {dim}}W.
\end{equation*}
\end{proposition}

\begin{proposition}
 Let $\beta \in {\mathrm {Hom}}(V\times V, W)$ be a flat bilinear form with respect to a Lorentzian inner product of $W$. Suppose
 that ${\mathrm {dim}}V>{\mathrm {dim}}W$ and $\beta (x,x)\neq 0$ for all $x \in V, x\neq 0$. Then there exist a non-zero isotropic 
vector $e \in W$, and a real valued bilinear form $\phi \in {\mathrm {Hom}}(V\times V, \mathbb{R})$ such that 
\begin{equation*}
 {\mathrm {dim}}N(\beta-\phi e)\geq {\mathrm {dim}}V-{\mathrm {dim}}W+2.
\end{equation*}
\end{proposition}

The following lemma is needed for the proof of the algebraic auxiliary results.

\begin{lemma}
Let  $\beta \in {\mathrm {Hom}}(V\times V, W)$ be a bilinear form, where $V$ and $W$ are both equipped with positive definite inner products
 and ${\mathrm {dim}}W\leq {\mathrm {dim}}V -2$. If $\beta \bullet \beta = k\langle.,.\rangle \bullet \langle.,.\rangle $ for some $k\neq0$,
 then $k>0$ and there exist a unit vector $\xi \in W$ and a subspace $V_{1}\subseteq V$  such that 
 $${\mathrm {dim}}V_{1}\geq {\mathrm {dim}}V-{\mathrm {dim}}W+1$$
 and 
$$\beta(x,y)=\sqrt{k}\langle x,y\rangle \xi {\ }{\ }\text{for
 all x} \in V {\ }{\mbox{and}} {\ } y \in V_{1}.$$ 
\end{lemma}

\begin{proof}
From $\beta \bullet \beta = k\langle.,.\rangle \bullet \langle.,.\rangle $, we easily deduce that the bilinear form
 $\widetilde{\beta} \in {\mathrm {Hom}}(V\times V, W\oplus\mathbb{R})$ given by 
\begin{eqnarray*}
\widetilde{\beta}(x,y):=\big( {\beta}(x,y),   \langle x,y \rangle     \big), {\ }{\ }x,y \in V,
\end{eqnarray*}
is flat with respect to the inner product $\langle\langle .,. \rangle\rangle$ of $W\oplus\mathbb{R}$ defined by 
\begin{eqnarray*}
\langle\langle (\xi,t), (\eta,s)\rangle\rangle:= \langle \xi,\eta \rangle-kst, {\ }{\ }(\xi,t), (\eta,s) \in W\oplus\mathbb{R}.
\end{eqnarray*}
We claim that $k>0$. Arguing indirectly, we suppose that $k<0$. Then $\langle\langle .,. \rangle\rangle$ 
is positive definite and according to Proposition 2.1, we obtain 
\begin{equation*}
 {\mathrm {dim}}N(\widetilde{\beta})\geq {\mathrm {dim}}V-{\mathrm {dim}}(W\oplus\mathbb{R}),
\end{equation*}
which contradicts our assumption on the dimensions.

Thus $k>0$, the inner product $\langle\langle .,. \rangle\rangle$ has Lorentzian signature and $\widetilde{\beta}$ fulfills the assumptions
 of Proposition 2.2. 
Hence there exist a non-zero isotropic vector $\eta=(e,t) \in W\oplus \mathbb{R}$ and a symmetric bilinear form
 $\phi\in {\mathrm {Hom}}(V\times V, \mathbb{R})$  such that
\begin{equation*}
 {\mathrm {dim}}N(\widetilde{\beta}-\phi \eta)\geq {\mathrm {dim}}V-{\mathrm {dim}}W+1.
\end{equation*}
Setting $V_1:=N(\widetilde{\beta}-\phi \eta)$, we immediately see that $\langle x,y\rangle=t\phi(x,y)$ and $\beta(x,y)=\phi(x,y)e$ 
for all $x \in V {\ }{\mbox{and}} {\ } y \in V_{1}$. Using the fact that $\eta$ is isotropic, we obviously obtain 
$\beta(x,y)=\sqrt{k}\langle x,y\rangle \xi$ for all 
$x \in V {\ }{\mbox{and}} {\ } y \in V_{1}$, where $\xi=\pm e/|e|$. 
\end{proof}

Hereafter, $V,W$ will be
 real vector spaces of dimensions $n$ and $p$ respectively, both equipped with positive definite inner products.
For each $\beta \in {\mathrm {Hom}}(V\times V, W)$, we define the map 
$$\beta^{\sharp}:W\longrightarrow {\mathrm {End}}(V), \; \xi\longmapsto \beta^{\sharp}(\xi)$$ 
such that 
$$\langle \beta^{\sharp}(\xi)x,y\rangle = \langle \beta(x,y),\xi \rangle, \; \text{for all} \; x,y \in V.$$
 Here, ${\mathrm {End}}(V)$ denotes the set of all selfadjoint endomorphisms of $(V, \langle,\rangle)$.

When $2\leq p\leq n/2$, for each $\beta \in {\mathrm {Hom}}(V\times V, W)$, we denote by $\Omega (\beta)$ the 
 following subset of the unit $(p-1)$-sphere $S^{p-1}$ in $W$
\begin{eqnarray*}
\Omega (\beta):= \big\lbrace u \in S^{p-1}: p\leq {\mbox {Index}} ( \beta^{\sharp}(u) ) \leq n-p  \big \rbrace.
\end{eqnarray*}

Moreover, we define the ``scalar curvature'' function ${\mathrm {sc}}:{\mathrm {Hom}}(V\times V, W)\longrightarrow \mathbb{R}$ by 
\begin{eqnarray*}
{\mathrm {sc}}(\beta):=\dfrac{1}{2}\sum_{i,j=1}^{n}\beta \bullet \beta (e_{i},e_{j},e_{i},e_{j}), 
\end{eqnarray*}
where $\lbrace e_{1},...,e_{n}\rbrace$ is an arbitrary orthonormal basis of $V$.

\smallskip
We now may state the auxiliary results that are crucial for the proofs of the main results.

\begin{proposition}
Given positive integers $2\leq p\leq n/2 $ and numbers $k \in \mathbb{R}, \delta >0$, there exists a constant $c=c(n,p,k,\delta)>0$,
 such that  the following inequality holds 
\begin{equation*}
\big\Vert  \beta \bullet \beta -k \langle.,.\rangle \bullet \langle.,.\rangle  \big\Vert ^{2}  \geq c(n,p,k,\delta)
 \Big (\int_{\Omega_k (\beta)} \vert \det \beta^{\sharp}(u)\vert dS_{u}\Big )^{4/n},
\end{equation*}
for any $\beta \in {\mathrm {Hom}}(V\times V, W)$ with $|\mathrm{sc}(\beta)|\geq \delta^2\Vert\beta\Vert^{2}$, where
$$
\Omega_k (\beta):=\left\{\begin{array}{l}\Omega (\beta)\;\;\;\;\;\;\;\;\mbox{if}\;\;k>0,
\vspace*{1.5ex}\\
S^{p-1}\;\;\;\;\;\;\;\; \mbox{if}\;\;k\leq 0.
\end{array} \right.
$$
\end{proposition}

\begin{proof}
 We consider the functions $\varphi_{k},\psi_k:{\mathrm {Hom}}(V\times V, W)\longrightarrow \mathbb{R}$ defined by
 \begin{equation*}
\varphi_{k}(\beta):=\big\Vert  \beta \bullet \beta -k \langle.,.\rangle \bullet \langle.,.\rangle  \big\Vert ^{2}, {\ }{\ }   \psi_k(\beta):
=\int_{\Omega_k (\beta)} \vert \det \beta^{\sharp}(u)\vert dS_{u},
 \end{equation*}
where $\beta \in {\mathrm {Hom}}(V\times V, W)$. Moreover, we define the function $\omega_{k}:U_{k,\delta} \longrightarrow \mathbb{R}$ by 
\begin{equation*}
  \omega_{k}(\beta):=\dfrac{\varphi_{k}(\beta)}{\big(\psi_k(\beta)\big)^{4/n}},{\ } \beta \in U_{k,\delta}, 
\end{equation*}
where 
$$
U_{k,\delta}:=\left\{\begin{array}{l}\Big\lbrace \beta \in {\mathrm {Hom}}(V\times V, W): \psi_k(\beta)\neq 0 {\ }{\ }{\mbox{and}}
 {\ }{\ } |\mathrm{sc}(\beta)|\geq \delta^2\Vert\beta\Vert^{2} \Big\rbrace\;\;\;\;\;\mbox{if}\;\;k\neq0,
\vspace*{1.5ex}\\
\Big\lbrace \beta \in {\mathrm {Hom}}(V\times V, W): \psi_k(\beta)=1 {\ }{\ }{\mbox{and}}
 {\ }{\ } |\mathrm{sc}(\beta)|\geq \delta^2\Vert\beta\Vert^{2} \Big\rbrace\;\;\;\;\; \mbox{if}\;\;k=0.
\end{array} \right.
$$

We shall prove that $\inf \omega_{k}(U_{k,\delta})>0$. Arguing indirectly, we assume that
there is  a sequence $\{\beta_m\}$  in $U_{k,\delta}$ such that 
\begin{equation}
\lim _{m\rightarrow \infty}\omega_{k}(\beta_{m})=0.
\end{equation}
We note that $\beta_{m}\neq 0$ for all $m \in \mathbb{N}$, since $\beta_{m} \in U_{k,\delta}$. 
Then we may write $\beta_{m}=\Vert\beta_{m}\Vert \widehat{\beta}_{m}$, where $ \Vert \widehat{\beta}_{m} \Vert =1$.

To reach a contradiction, we distinguish two cases.

\smallskip

\textit{Case 1}. Suppose that the sequence $\{\beta_m\}$ is unbounded. We may assume, by taking a subsequence if necessary,
 that $\lim _{m\rightarrow \infty}\Vert\beta_{m}\Vert =+\infty$.
 Since $ \Vert \widehat{\beta}_{m} \Vert =1$, we may also assume 
that $\{  \beta_m\}$ converges to some  $\widehat{\beta} \in {\mathrm {Hom}}(V\times V, W)$ 
with $\Vert \widehat{\beta} \Vert=1$.

Using  the fact that $\psi_k$ is homogeneous of degree $n$, (2.1) yields 
\begin{equation*}
\lim _{m\rightarrow \infty}\dfrac{\|\widehat{\beta}_m\bullet \widehat{\beta}_m-\frac{k}{\|\beta_{m}\|^2} \langle,\rangle\bullet  \langle,\rangle\|^2}
{(\psi_k(\widehat{\beta}_{m}))^{4/n}}=0.
\end{equation*} 
Since $\{\psi_k(\widehat{\beta}_{m}) \}$ is bounded, the above implies that  $\widehat{\beta}$ is flat. On the other hand, 
from ${\beta}_m \in U_{k,\delta}$ we have $|\mathrm{sc}(\widehat{\beta}_m)|\geq \delta^2$, and taking the limit for $m\rightarrow \infty$ we obtain 
$|\mathrm{sc}(\widehat{\beta})|\geq \delta^2$, contradiction.

\smallskip
\textit{Case 2}. Assume that the sequence $\{\beta_m\}$ is bounded. Then $\{\beta_m\}$  converges to some
 ${\beta} \in {\mathrm {Hom}}(V\times V, W)$, 
by taking a subsequence if necessary. 
From (2.1) it follows that $\varphi_{k}({\beta})=0$,
 or  equivalently ${\beta} \bullet {\beta} = k\langle.,.\rangle \bullet \langle.,.\rangle $. 

At first we assume that $k=0$. Then $\beta$ is flat and non-zero.  In fact, if ${\beta}=0$, then  
${\beta}^{\sharp}(u)=0$ for all $u \in S^{p-1}$. Since $\beta_{m} \in U_{0,\delta}$, there exists $\xi_{m} \in S^{p-1}$ such that 
\begin{equation}
 \vert \det {\beta}^{\sharp}_{m}(\xi_{m})\vert {\mbox {Vol}} (S^{p-1} )=1 {\ }{\ }{\mbox {for all}}{\ }{\ } m \in \mathbb{N}.
\end{equation} 
On account of $\vert\xi_{m}\vert=1$, we may assume that the sequence $\lbrace \xi_{m} \rbrace$ converges to some $\xi \in S^{p-1}$, 
by passing to a subsequence if necessary. Then, from $\lim _{m\rightarrow \infty}\beta_{m}={\beta}$, we get 
$\lim _{m\rightarrow \infty}{\beta}^{\sharp}_{m}(\xi_{m})={\beta}^{\sharp}(\xi)=0$, which contradicts (2.2). Thus $\beta \neq 0$. On the other hand, 
from ${\beta}_m \in U_{k,\delta}$ we have $|\mathrm{sc}({\beta}_m)|\geq \delta^2\|{\beta}_m\|^2$. Taking the limit for $m\rightarrow \infty$,  
we obtain 
$|\mathrm{sc}({\beta})|\geq \delta^2\|{\beta}\|^2$, contradiction since $\beta$ is flat.

Now assume that $k\neq 0$. According to Lemma 2.1, $k>0$ and there exist a unit vector $\xi \in W$ and a subspace $V_{1}\subseteq V$  such that 
 ${\mathrm {dim}}V_{1}\geq n-p+1$
 and 
\begin{eqnarray}
{\beta}(x,y)=\sqrt{k}\langle x,y\rangle \xi {\ }{\ }{\mbox {for all}}{\ }{\ } x \in V {\ }{\mbox{and}} {\ } y \in V_{1}. 
\end{eqnarray}

By virtue of the fact that $\beta_{m} \in U_{k,\delta}$, there exists an open subset $\mathcal{U}_{m}$ of $S^{p-1}$ such that 
$$\mathcal{U}_{m}\subseteq  \big\lbrace u \in S^{p-1}: p \leq {\mbox {Index}} ( \beta^{\sharp}_{m}(u)) \leq n-p  \big \rbrace$$
and 
$$\det \beta^{\sharp}_{m}(u)\neq0 \;\; \text{for all}\;\; u \in \mathcal{U}_{m} \;\;\text{and}\;\; m \in \mathbb{N}.$$

Let $\lbrace u_{m}\rbrace$ be a sequence such that $u_{m} \in \mathcal{U}_{m}$ for all $m \in \mathbb{N}$. On account of $\vert u_{m} \vert =1$,
 we may assume hereafter that $\lbrace u_{m}\rbrace$ is convergent, by passing if necessary to a subsequence. 
We set $u=\lim _{m\rightarrow \infty}u_{m}$. Since $\lim _{m\rightarrow \infty}\beta^{\sharp}_{m}(u_{m})={\beta}^{\sharp}(u)$ 
and $u_{m} \in \mathcal{U}_{m}$, we deduce that ${\mbox {Index}} ({\beta}^{\sharp}(u)) \leq n-p$. 
Then from (2.3) 
we get $\langle u,\xi\rangle  \geq 0.$ We claim that $\langle u,\xi\rangle  = 0$. Indeed, if $\langle u,\xi\rangle  > 0,$ then (2.3) 
implies that ${\beta}^{\sharp}(u)$ has at least $n-p+1$ positive eigenvalues, and so, 
for $m$ 
large enough, $\beta^{\sharp}_{m}(u_{m})$ has at least $n-p+1$ positive eigenvalues.
This, on account of the fact that $\det \beta^{\sharp}_{m}(u)\neq0$ for all $u \in \mathcal{U}_{m}$,  confirms that $\beta^{\sharp}_{m}(u_{m})$ has 
at most $p-1$ 
negative eigenvalues, that is, ${\mbox {Index}} (\beta^{\sharp}_{m}(u_{m})) \leq p-1$, contradiction, since $u_{m} \in \mathcal{U}_{m}$.

Thus, we have proved that for any convergent sequence $\lbrace u_{m}\rbrace$ such that $u_{m} \in \mathcal{U}_{m}$ for all $m \in \mathbb{N}$, 
we have $\langle \lim _{m\rightarrow \infty}u_{m}, \xi \rangle =0$.

Since $\mathcal{U}_{m}$ is open, we may choose convergent sequences $\{u^{(1)}_{m}\},\{u^{(2)}_{m}\},...,\{u^{(p)}_{m}\}$ 
such that ${u^{(1)}_{m}},{u^{(2)}_{m}},...,{u^{(p)}_{m}} \in \mathcal{U}_{m}$ and  span $W$ for all $m \in \mathbb{N}$. Then, 
by virtue of (2.3) and the fact that $\langle \lim _{m\rightarrow \infty}u^{(\alpha)}_{m}, \xi \rangle =0$ for 
all $\alpha \in \lbrace1,2,...,p\rbrace$, we infer that the restriction of $\beta_{m}$ to $V_{1}\times V_{1}$ 
satisfies 
$$\lim _{m\rightarrow \infty}\beta_{m}\big\vert_{V_{1}\times V_{1}}=0$$
and consequently
\begin{equation}
\lim _{m\rightarrow \infty}\big(\beta_{m}\bullet\beta_{m}\big)\Big\vert_{V_{1}\times V_{1}\times V_{1}\times V_{1}}=0. 
\end{equation} 
From the obvious inequality 
$$\Big\Vert\Big(\beta_{m}\bullet\beta_{m}-k\langle.,.\rangle\bullet\langle.,.\rangle\Big)\Big\vert_{V_{1}\times V_{1}\times 
V_{1}\times V_{1}}\Big\Vert^{2}\leq \varphi_{k}(\beta_{m}),$$
(2.4) and by virtue of $\lim _{m\rightarrow \infty}\varphi_{k}(\beta_{m})=\varphi_{k}({\beta})=0$, we reach a contradiction since $k>0$. 

Thus, we have proved that $\inf \omega_{k}(U_{k,\delta})>0$. 
 Obviously  $\inf \omega_{k}(U_{\delta})$ depends only on $n,p,k$ and $\delta$ and is denoted by $c(n,p,k,\delta)$. Then the 
desired inequality follows immediately.
\end{proof}

\begin{proposition}
Given positive integers $2\leq  p\leq n/2 $ and a number $\delta >0$, there exists a constant $c=c(n,p,\delta)>0$, 
such the following inequality holds 
\begin{equation*}
\Big\Vert  \beta \bullet \beta -\dfrac{{\mathrm {sc}}(\beta)}{n(n-1)} \langle.,.\rangle \bullet \langle.,.\rangle  \Big\Vert ^{2} 
 \geq c(n,p,\delta) \Big (\int_{\Lambda (\beta)} \vert \det \beta^{\sharp}(u)\vert dS_{u}\Big )^{4/n},
\end{equation*}
for any $\beta \in {\mathrm {Hom}}(V\times V, W)$ with 
$|{\mathrm {sc}}(\beta)|\geq \delta^{2} \Vert\beta\Vert^{2}$, where
$$
\Lambda (\beta):=\left\{\begin{array}{l}\Omega (\beta)\;\;\;\;\;\;\;\;\mbox{if}\;\;{\mathrm {sc}}(\beta)>0,
\vspace*{1.5ex}\\
S^{p-1}\;\;\;\;\;\;\; \;\mbox{if}\;\;{\mathrm {sc}}(\beta)\leq 0.
\end{array} \right.
$$
\end{proposition}

\begin{proof}
 We consider the functions $\varphi,\psi:{\mathrm {Hom}}(V\times V, W)\longrightarrow \mathbb{R}$ given by
 \begin{equation*}
\varphi(\beta):=\Big\Vert  \beta \bullet \beta -\dfrac{{\mathrm {sc}}(\beta)}{n(n-1)} \langle.,.\rangle \bullet \langle.,.\rangle 
 \Big\Vert ^{2}, {\ }{\ }   \psi(\beta):=\int_{\Lambda(\beta)} \vert \det \beta^{\sharp}(u)\vert dS_{u},
 \end{equation*}
where $\beta \in {\mathrm {Hom}}(V\times V, W)$. 

We shall prove that $\varphi$ attains a positive minimum on $U_{\delta}$,
where 
\begin{equation*}
U_{\delta}:=\Big\lbrace \beta \in {\mathrm {Hom}}(V\times V, W): \psi(\beta)=1 {\ }{\ }{\mbox{and}} {\ }{\ }|{\mathrm {sc}}(\beta)|\geq 
\delta^{2} \Vert\beta\Vert^{2}\Big \rbrace. 
\end{equation*}
There exists a sequence $ \lbrace \beta_{m} \rbrace$  in $U_{\delta}$ such that 
\begin{equation*}
\lim _{m\rightarrow \infty}\varphi(\beta_{m})=\inf \varphi (U_{\delta}).
\end{equation*}
We observe that $\beta_{m}\neq 0$ for all $m \in \mathbb{N}$, since  $\beta_{m} \in U_{\delta}$. 
Then we may write $\beta_{m}=\Vert\beta_{m}\Vert \widehat{\beta}_{m}$, where $ \Vert \widehat{\beta}_{m} \Vert =1$. 

We claim that the sequence $ \lbrace \beta_{m} \rbrace$ 
is bounded. Assume to the contrary that there exists a subsequence of $ \lbrace \beta_{m} \rbrace$, which by abuse
 of notation is denoted again by $ \lbrace \beta_{m} \rbrace$, such that $\lim _{m\rightarrow \infty}\Vert\beta_{m}\Vert =+\infty$.
 Since $ \Vert \widehat{\beta}_{m} \Vert =1$, we may assume, by taking a subsequence if necessary, 
that $\{\widehat
{\beta}_m \}$ converges to some  $\widehat{\beta} \in {\mathrm {Hom}}(V\times V, W)$ 
with $\Vert \widehat{\beta} \Vert=1$. Using the fact that $\psi$ is homogeneous of degree $n$ and since $ \lbrace \beta_{m} \rbrace \in U_{\delta}$, 
we get
\begin{equation*}
\Vert\beta_{m}\Vert=\dfrac{1}{(\psi(\widehat{\beta}_{m}))^{{1}/{n}}}.
\end{equation*} 
Thus $\lim _{m\rightarrow \infty}\psi(\widehat\beta_{m})=0$ and consequently  $\varphi(\widehat{\beta})=0$, or  equivalently 
$$\widehat{\beta} \bullet \widehat{\beta} = \dfrac{{\mathrm {sc}}(\widehat{\beta})}{n(n-1)}\langle.,.\rangle \bullet \langle.,.\rangle .$$
From $ \lbrace \beta_{m} \rbrace  \in U_{\delta}$, we have $|{\mathrm {sc}}(\widehat{\beta}_{m})|\geq \delta^{2}$. Taking the
 limit for $m\rightarrow\infty$, we get ${\mathrm {sc}}(\widehat{\beta})\neq0$. According to Lemma 2.1, ${\mathrm {sc}}(\widehat{\beta})>0$ 
and there exists a unit vector $\widehat{\xi} \in W$ and a subspace
 $\widehat{V}_{1}\subseteq V$ such that ${\mathrm {dim}}\widehat{V}_{1}\geq n-p+1$ and
\begin{eqnarray}
\widehat{\beta}(x,y)=\Big(\dfrac{{\mathrm {sc}}(\widehat{\beta})}{n(n-1)}\Big)^{{1}/{2}}\langle x,y\rangle \widehat{\xi}
 {\ }{\ }{\mbox {for all}}{\ }{\ } x \in V {\ }{\mbox{and}} {\ } y \in \widehat{V}_{1}. 
\end{eqnarray} 
Moreover, since $\beta_m \in U_{\delta},$  there exists an open subset $\widehat{\mathcal{U}}_{m}$ of $S^{p-1}$ such that
 $$\widehat{\mathcal{U}}_{m}\subseteq \Lambda (\widehat{\beta}_m)\;\; \text{and} \;\;\det \widehat{\beta}^{\sharp}_{m}(u)\neq0 
\;\;\text{for all} \;\;u \in \widehat{\mathcal{U}}_{m} \;\;\text{and} \;\;m \in \mathbb{N}.$$
 From ${\mathrm {sc}}(\widehat{\beta})>0$, we deduce that
 ${\mathrm {sc}}(\widehat{\beta}_m)>0$ and so
$$\widehat{\mathcal{U}}_{m}\subseteq  \big\lbrace u \in S^{p-1}: p\leq {\mbox {Index}} ( \widehat{\beta}^{\sharp}_{m}(u))
 \leq n-p  \big \rbrace$$
for $m$ large enough.

Let $\lbrace \widehat{u}_{m}\rbrace$ be a sequence such that $\widehat{u}_{m} \in \widehat{\mathcal{U}}_{m}$ for all $m \in \mathbb{N}$. 
On account of $\vert \widehat{u}_{m} \vert =1$, we may assume hereafter that $\lbrace \widehat{u}_{m}\rbrace$ is convergent, 
by passing if necessary to a subsequence. We set $\widehat{u}=\lim _{m\rightarrow \infty}\widehat{u}_{m}$. 
Since  $\lim _{m\rightarrow \infty}\widehat{\beta}^{\sharp}_{m}(\widehat{u}_{m})={\widehat{\beta}}^{\sharp}(\widehat{u})$ and 
  $\widehat{u}_{m} \in \widehat{\mathcal{U}}_{m}$, we deduce that
 ${\mbox {Index}} ( \widehat{\beta}^{\sharp}(\widehat{u}) ) \leq n-p$. Then, from (2.5)  we obtain $\langle \widehat{u},\widehat{\xi}\rangle  \geq 0.$ 
We claim that $\langle \widehat{u},\widehat{\xi}\rangle  = 0.$ Indeed, if $\langle \widehat{u},\widehat{\xi}\rangle  > 0,$ 
then (2.5) implies that $\widehat{\beta}^{\sharp}(\widehat{u})$ 
has at least $n-p+1$ 
positive eigenvalues, and so, for $m$ large enough, $\beta^{\sharp}_{m}(u_{m})$ has at least $n-p+1$ positive eigenvalues.
This, on account of the fact that $\det \widehat{\beta}^{\sharp}_{m}(u)\neq0$ for all $u \in \widehat{\mathcal{U}}_{m}$, confirms
 that $\widehat{\beta}^{\sharp}_{m}(\widehat{u}_{m})$ has at most $p-1$ negative eigenvalues, that is, ${\mbox {Index}}
 (\widehat{\beta}^{\sharp}_{m}(\widehat{u}_{m})) \leq p-1$, contradiction, since $\widehat{u}_{m} \in \widehat{\mathcal{U}}_{m}$.

Thus, we have proved that for any convergent sequence $\lbrace \widehat{u}_{m}\rbrace$ such that $\widehat{u}_{m} \in \widehat{\mathcal{U}}_{m}$ for
 all $m \in \mathbb{N}$, we have $\langle \lim _{m\rightarrow \infty}\widehat{u}_{m}, \widehat{\xi} \rangle =0$.

Since $\widehat{\mathcal{U}}_{m}$ is open, we may choose convergent sequences $\{\widehat{u}^{(1)}_{m}\},\{\widehat{u}^{(2)}_{m}\},...,
\{\widehat{u}^{(p)}_{m}\}$ such that ${\widehat{u}^{(1)}_{m}},{\widehat{u}^{(2)}_{m}},...,{\widehat{u}^{(p)}_{m}} \in \widehat{\mathcal{U}}_{m}$ 
and  span $W$ for all $m \in \mathbb{N}$. Then, by virtue of (2.5) and the fact that $\langle \lim _{m\rightarrow \infty}\widehat{u}^{(\alpha)}_{m},
 \widehat{\xi} \rangle =0$ for all $\alpha \in \lbrace1,2,...,p\rbrace$, we infer that the restriction 
 of $\widehat{\beta}_{m}$ to $\widehat{V}_{1}\times \widehat{V}_{1}$ satisfies 
$$\lim _{m\rightarrow \infty}\widehat{\beta}_{m}\big\vert_{\widehat{V}_{1}\times \widehat{V}_{1}}=0$$
and consequently
\begin{equation}
\lim _{m\rightarrow \infty}\big(\widehat{\beta}_{m}\bullet\widehat{\beta}_{m}\big)\Big\vert_{\widehat{V}_{1}\times \widehat{V}_{1}
\times \widehat{V}_{1}\times \widehat{V}_{1}}=0. 
\end{equation} 
From the  inequality 
$$\Big\Vert\Big(\widehat{\beta}_{m}\bullet\widehat{\beta}_{m}-\dfrac{{\mathrm {sc}}(\widehat{\beta}_{m})}{n(n-1)}\langle.,.
\rangle\bullet\langle.,.\rangle\Big)\Big\vert_{\widehat{V}_{1}\times \widehat{V}_{1}\times \widehat{V}_{1}\times 
\widehat{V}_{1}}\Big\Vert^{2}\leq \varphi(\widehat{\beta}_{m}),$$
(2.6) and by virtue of $\lim _{m\rightarrow \infty}\varphi(\widehat{\beta}_{m})=\varphi(\widehat{\beta})=0$, we reach a 
contradiction since ${\mathrm {sc}}(\widehat{\beta})\geq\delta^{2}>0$.

Thus, the sequence $ \lbrace \beta_{m} \rbrace$ is bounded, and it converges to some ${\beta} \in {\mathrm {Hom}}(V\times V, W)$, 
by taking a subsequence if necessary. 

 We claim that $\varphi({\beta})>0.$ Assume to the contrary that $\varphi({\beta})=0$, or  
equivalently 
$${\beta} \bullet {\beta} = \dfrac{{\mathrm {sc}}({\beta})}{n(n-1)}\langle.,.\rangle \bullet \langle.,.\rangle .$$
We notice that ${\beta} \neq0$. Indeed if ${\beta}=0$, then 
${\beta}^{\sharp}(u)=0$ for all $u \in S^{p-1}$. Since $\beta_{m} \in U_{\delta}$, there exists $\xi_{m} \in \Lambda (\beta_{m})$ such that 
\begin{equation}
 \vert \det {\beta}^{\sharp}_{m}(\xi_{m})\vert {\mbox {Vol}} ( \Lambda (\beta_{m}) )=1 {\ }{\ }{\mbox {for all}}{\ }{\ } m \in \mathbb{N}.
\end{equation} 
On account of $\vert\xi_{m}\vert=1$, we may assume that the sequence $\lbrace \xi_{m} \rbrace$ converges to some $\xi \in S^{p-1}$, 
by passing to a subsequence if necessary. Then, from $\lim _{m\rightarrow \infty}\beta_{m}={\beta}$, we get 
$\lim _{m\rightarrow \infty}{\beta}^{\sharp}_{m}(\xi_{m})={\beta}^{\sharp}(\xi)=0$, which contradicts (2.7). 

Therefore ${\beta} \neq0$. From the fact that $\beta_{m} \in U_{\delta}$, we have 
$|{\mathrm {sc}}(\beta_{m})|\geq\delta^{2}\Vert\beta_{m} \Vert ^{2}$. Taking the limit for $m\rightarrow\infty$, 
we deduce that $|{\mathrm {sc}}({\beta})|\geq\delta^{2}\Vert {\beta} \Vert ^{2}>0$.

Then, according to Lemma 2.1, there exists a unit vector ${\xi} \in W$ and a subspace ${V}_{1}\subseteq V$ such that
 ${\mathrm {dim}}{V}_{1}\geq n-p+1$ and
\begin{eqnarray}
{\beta}(x,y)
=\Big({\dfrac{{\mathrm {sc}}({\beta})}{n(n-1)}} \Big)^{{1}/{2}}\langle x,y\rangle \xi 
{\ }{\ }{\mbox {for all}}{\ }{\ } x \in V {\ }{\mbox{and}} {\ } y \in {V}_{1}. 
\end{eqnarray}

By virtue of the fact that $\beta_{m} \in U_{\delta}$, there exists an open subset $\mathcal{U}_{m}$ of $S^{p-1}$ such that 
$$\mathcal{U}_{m}\subseteq \Lambda (\beta_m)\;\; \text{and} \;\;\det \beta^{\sharp}_{m}(u)\neq0\;\; \text{for all} \;\;u
 \in \mathcal{U}_{m}\;\; \text{and} \;\;m \in \mathbb{N}.$$
Since ${\mathrm {sc}}({\beta})>0$, we see that
$$\mathcal{U}_{m}\subseteq  \big\lbrace u \in S^{p-1}: p\leq {\mbox {Index}} ( \beta^{\sharp}_{m}(u)) \leq n-p  \big \rbrace$$
for $m$ large enough.

Let $\lbrace u_{m}\rbrace$ be a sequence such that $u_{m} \in \mathcal{U}_{m}$ for all $m \in \mathbb{N}$. On account of $\vert u_{m} \vert =1$, 
we may assume hereafter that $\lbrace u_{m}\rbrace$ is convergent, by passing if necessary to a subsequence. 
We set $u=\lim _{m\rightarrow \infty}u_{m}$. Since $\lim _{m\rightarrow \infty}{\beta}^{\sharp}_{m}(u_{m})={\beta}^{\sharp}(u)$ and 
$u_{m} \in \mathcal{U}_{m}$, we deduce that ${\mbox {Index}} ({\beta}^{\sharp}(u)) \leq n-p$. 
Then, from (2.8), we get $\langle u,\xi\rangle  \geq 0.$ 
We claim that $\langle u,\xi\rangle  = 0.$ Indeed, if $\langle u,\xi\rangle  > 0,$ then (2.8) implies that
 ${\beta}^{\sharp}(u)$ has at least $n-p+1$ positive eigenvalues, and so, for $m$ 
large enough, ${\beta}^{\sharp}_{m}(u_{m})$ has at least $n-p+1$ positive eigenvalues.
This, on account of the fact that $\det {\beta}^{\sharp}_{m}(u)\neq0$ for all $u \in \mathcal{U}_{m}$, confirms that ${\beta}^{\sharp}_{m}(u_{m})$ 
has at most $p-1$
negative eigenvalues, that is, ${\mbox {Index}} ({\beta}^{\sharp}_{m}(u_{m})) \leq p-1$, contradiction, since $u_{m} \in \mathcal{U}_{m}$.

Thus, we have proved that for any convergent sequence $\lbrace u_{m}\rbrace$ such that $u_{m} \in \mathcal{U}_{m}$ for all $m \in \mathbb{N}$, 
we have $\langle \lim _{m\rightarrow \infty}u_{m}, \xi \rangle =0$.

Since $\mathcal{U}_{m}$ is open, we may choose convergent sequences $\{u^{(1)}_{m}\},\{u^{(2)}_{m}\},...,\{u^{(p)}_{m}\}$ 
such that ${u^{(1)}_{m}},{u^{(2)}_{m}},...,{u^{(p)}_{m}} \in \mathcal{U}_{m}$ and  span $W$ for all $m \in \mathbb{N}$. 
Then, by virtue of (2.8) and the fact that $\langle \lim _{m\rightarrow \infty}u^{(\alpha)}_{m}, \xi \rangle =0$ for 
all $\alpha \in \lbrace1,2,...,p\rbrace$, we infer that the restriction of $\beta_{m}$ to $V_{1}\times V_{1}$ 
satisfies 
$$\lim _{m\rightarrow \infty}\beta_{m}\big\vert_{V_{1}\times V_{1}}=0$$
and consequently
\begin{equation}
\lim _{m\rightarrow \infty}\big(\beta_{m}\bullet\beta_{m}\big)\Big\vert_{V_{1}\times V_{1}\times V_{1}\times V_{1}}=0. 
\end{equation} 
From the inequality 
$$\Big\Vert\Big(\beta_{m}\bullet\beta_{m}-\dfrac{{\mathrm {sc}}({\beta}_{m})}{n(n-1)}\langle.,.
\rangle\bullet\langle.,.\rangle\Big)\Big\vert_{V_{1}\times V_{1}\times V_{1}\times V_{1}}\Big\Vert^{2}\leq \varphi(\beta_{m}),$$
(2.9) and by virtue of $\lim _{m\rightarrow \infty}\varphi(\beta_{m})=\varphi({\beta})=0$, 
we reach a contradiction since ${\mathrm {sc}}({\beta})\geq\delta^{2}\Vert {\beta} \Vert ^{2}>0$.

Consequently, our claim is proved, that is, $\varphi({\beta})>0$ and so  $\varphi$ attains a positive minimum
 on $U_{\delta}$ which obviously depends only on $n,p,\delta$ and is denoted by $c(n,p,\delta)$. 

Now let $\beta\in U_{\delta}$. Assume that $\psi(\beta)\neq 0$ and set $\widetilde{\beta}=\beta/(\psi(\beta))^{{1}/{n}}$. 
Clearly $\widetilde{\beta}\in U_{\delta}$, and consequently $\varphi(\widetilde{\beta})\geq c(n,p,\delta)$. Since $\varphi$ is 
homogeneous of degree $4$, the desired inequality is obviously fulfilled. In the case where $\psi(\beta)=0$, the  inequality is trivial.
\end{proof}

Now we argue on the necessity of the assumption on the scalar curvature of bilinear forms in both Propositions 2.3 and 2.4.
Actually we provide counterexamples that ensure that this assumption cannot be dropped. 

\begin{example}
Let $\{\eta_m\}, \{a^{(j)}_{m}\}, 2 \leq j\leq n, \{b^{(\alpha)}_{m}\}, 2 \leq \alpha \leq p, $ be sequences of real  numbers that tend to zero as 
$m\rightarrow \infty$. Furthermore, we consider convergent sequences $\{\gamma^{(1)}_{m}\}, \{\theta^{(j,\alpha)}_m  \},
 2 \leq j\leq n, 2 \leq \alpha \leq p,$ such that
$$\eta_m^{{2(n-1)}/{n}}\sum_{j=2}^n(a^{(j)}_{m})^2+\sum_{\alpha=2}^p\sum_{j=2}^n(\theta^{(j,\alpha)}_m)^2=1$$
and
$$(\gamma^{(1)}_{m})^2=1-\eta_m^2-\eta_m^{{2(n-2)}/{n}}\sum_{\alpha=2}^p(b^{(\alpha)}_{m})^2.$$

We now define the sequence $\{\gamma_m\}$ in ${\mathrm {Hom}}(\mathbb{R}^n\times \mathbb{R}^n, \mathbb{R}^p)$ by 
\begin{eqnarray*}
 \gamma_m(x,y)=\big(\gamma^{(1)}_{m}x_1y_1+\eta_m^{{2(n-1)}/{n}}\sum_{j=2}^{n}a^{(j)}_{m}x_jy_j\big)\xi_1\;\;\;\;\;\;\;\;\;\;\;\;\;\;\;\;\\
+\eta_m^{(n-2)/{n}}
\sum_{\alpha=2}^{p} \big(b^{(\alpha)}_{m}x_1y_1+\eta_m^{{2}/{n}}\sum_{j=2}^{n}\theta^{(\alpha,j)}_{m}x_jy_j\big)\xi_{\alpha},
\end{eqnarray*}
where $x=(x_1,\dots,x_n), y=(y_1,\dots,y_n)$ and $\xi_1,\dots,\xi_p$ is the standard basis of $\mathbb{R}^p$. The sequences are chosen so that 
$$p\leq {\mbox {Index}}\; \Big({\mbox {diag}}(\gamma^{(1)}_{m},\eta_m^{{2(n-1)}/{n}}a^{(2)}_{m},\dots,\eta_m^{{2(n-1)}/{n}}a^{(n)}_{m})\Big)
 \leq n-p$$
 and 
$$p\leq {\mbox {Index}}\; \Big({\mbox {diag}}(b^{(\alpha)}_{m},\eta_{m}^{{2}/{n}}\theta^{(\alpha,2)}_{m},\dots,
\eta_{m}^{{2}/{n}}\theta^{(\alpha,n)}_{m})\Big) \leq n-p$$
 for any $2\leq \alpha \leq p.$ This implies that there exists an open subset $\Omega $ of $S^{p-1}$ such that $\Omega \subseteq \Omega (\gamma_m)$ 
for all $m \in \mathbb{N}$.

A direct computation shows that 
\begin{equation*}
 \Big\Vert  \gamma_m \bullet \gamma_m -\dfrac{{\mathrm {sc}}(\gamma_m)}{n(n-1)} \langle.,.\rangle \bullet \langle.,.\rangle  \Big\Vert ^{2}=
\eta_{m}^{{4(n-1)}/{n}}\rho_m,
\end{equation*}
where
\begin{eqnarray*}
\rho_m= 32 \sum_{j=2}^n 
\Big(\gamma^{(1)}_{m}a^{(j)}_{m}+\sum_{\alpha=2}^{p}b^{(\alpha)}_{m}\theta^{(j,\alpha)}_m
-\frac{2}{n(n-1)}\sum_{t=2}^n\gamma^{(1)}_{m}a^{(t)}_{m}\;\;\;\;\;\;\;\;\;\;\;\;\;\;\;\;\;\\
-\frac{2}{n(n-1)}\sum_{t=2}^n\sum_{\alpha=2}^{p}b^{(\alpha)}_{m}\theta^{(t,\alpha)}_m
-  \frac{\eta_{m}^{{2}/{n}}}{n(n-1)} \sum_{s\neq t, s,t\geq2} \theta^{(t,\alpha)}_m  \theta^{(s,\alpha)}_m  \Big)^2\\
+16\sum_{i\neq j, i,j \geq2}^n 
\Big(\eta_{m}^{{2}/{n}}  \sum_{\alpha=2}^{p}  \theta^{(i,\alpha)}_m \theta^{(j,\alpha)}_m
-  \frac{2}{n(n-1)}  \sum_{t=2}^n\gamma^{(1)}_{m}a^{(t)}_{m}\;\;\;\;\;\;\;\;\;\;\;
\;\;\;\;\;\\-  
\frac{2}{n(n-1)}  \sum_{t=2}^n\sum_{\alpha=2}^{p}b^{(\alpha)}_{m}\theta^{(t,\alpha)}_m
 -  \frac{\eta_{m}^{{2}/{n}}}{n(n-1)} \sum_{s\neq t, s,t\geq2} \theta^{(t,\alpha)}_m  \theta^{(s,\alpha)}_m  \Big)^2.
\end{eqnarray*}

Moreover, writing $u=\sum_{\alpha}u_{\alpha}\xi_{\alpha}$, we see that
$$\int_{\Omega} \vert \det \gamma_m^{\sharp}(u)\vert dS_{u}=
\eta_{m}^{n-1}\sigma_m,$$
where
$$\sigma_m=
\int_{\Omega} \Big\vert \big(u_{1}\gamma_m^{(1)} +\eta_m^{(n-2)/n}{\sum_{\alpha=2}^{p}u_{\alpha} 
 } \big)\prod^n_{j=2}\big(\eta_m^{(n-2)/n}   a_m^{(j)}   +   \sum_{\alpha=2}^{p} u_{\alpha} \theta^{(j,\alpha)}_m \big)
\Big\vert dS_{u}.$$
We observe that $\lim _{m\rightarrow \infty}\rho_m =0$ and
$$\lim _{m\rightarrow \infty}\sigma_m=\int_{\Omega} \Big\vert u_{1} 
  \prod^n_{j=2}\sum_{\alpha=2}^{p} u_{\alpha} \theta^{(j,\alpha)}
\Big\vert dS_{u},$$
where $\theta^{(j,\alpha)}=\lim _{m\rightarrow \infty}\theta^{(j,\alpha)}_m$. 
We may also choose the sequences so that 
$$\int_{\Omega} \Big\vert u_{1} 
  \prod^n_{j=2}\sum_{\alpha=2}^{p} u_{\alpha} \theta^{(j,\alpha)}
\Big\vert dS_{u}>0.$$
For instance, we may choose the sequences $\{\theta^{(j,\alpha)}_m  \},
 2 \leq j\leq n,$ such that
 $$\lim _{m\rightarrow \infty}(\theta^{(\alpha,2)}_{m}\dots\theta^{(\alpha,n)}_{m})\neq 0 \;\; \text{for a fixed} \;\; 2\leq \alpha \leq p.$$
Hence
\begin{equation*}
\lim _{m\rightarrow \infty} \dfrac{\Big\Vert  \gamma_m \bullet \gamma_m -\dfrac{{\mathrm {sc}}(\gamma_m)}{n(n-1)} \langle.,
.\rangle \bullet \langle.,.\rangle  \Big\Vert ^{2}}{\Big(     \int_{\Omega} \vert \det \gamma_m^{\sharp}(u)\vert dS_{u}  \Big)^{4/n}}=
\lim _{m\rightarrow \infty}\dfrac{\rho_m}{(\sigma_m)^{4/n}}=0,
\end{equation*}
which implies that 
\begin{equation}
\lim _{m\rightarrow \infty} \dfrac{\Big\Vert  \gamma_m \bullet \gamma_m -\dfrac{{\mathrm {sc}}(\gamma_m)}{n(n-1)} \langle.,
.\rangle \bullet \langle.,.\rangle  \Big\Vert ^{2}}{\Big(   
  \int_{\Omega(\gamma_m)} \vert \det \gamma_m^{\sharp}(u)\vert dS_{u}  \Big)^{4/n}}=0.
\end{equation}
Furthermore, we notice that $\|\gamma_m\|^2=1$ and $\lim _{m\rightarrow \infty}{\mathrm {sc}}(\gamma_m)=0.$

This together with (2.10) show that there exist no positive constant depending only on $n,p$ such that  the inequality in Proposition 2.4 
 hold without the condition on the scalar curvature.

\end{example}

\begin{example}
We consider a sequence $\{\gamma_m\}$ in ${\mathrm {Hom}}(\mathbb{R}^n\times \mathbb{R}^n, \mathbb{R}^p)$ as in the previous example.
Arguing as before, we similarly conclude that 
\begin{equation*}
\lim _{m\rightarrow \infty} \dfrac{\big\Vert  \gamma_m \bullet \gamma_m   \big\Vert ^{2}}{\Big(   
  \int_{\Omega(\gamma_m)} \vert \det \gamma_m^{\sharp}(u)\vert dS_{u}  \Big)^{4/n}}=0.
\end{equation*}
This proves that there exist no positive constant depending only on $n,p$ such that  the inequality in Proposition 2.3 
 hold  for $k=0$ without the condition on the scalar curvature.

Now let $k\neq0$. We choose all  sequences in Example 2.2 so that  $k{\mathrm {sc}}(\gamma_m)>0$ for all $m \in \mathbb{N}$, and
 we consider the sequence $\{\beta_m\}$ given by
$$\beta_m=\dfrac{n(n-1)k}{{\mathrm {sc}}(\gamma_m)}\gamma_m.$$
We have
\begin{equation*}
\dfrac{\big\Vert  \beta_m \bullet \beta_m -k \langle.,
.\rangle \bullet \langle.,.\rangle  \big\Vert ^{2}}{\Big(   
  \int_{\Omega_k(\beta_m)} \vert \det \beta_m^{\sharp}(u)\vert dS_{u}  \Big)^{4/n}}=
\dfrac{\Big\Vert  \gamma_m \bullet \gamma_m -\dfrac{{\mathrm {sc}}(\gamma_m)}{n(n-1)} \langle.,
.\rangle \bullet \langle.,.\rangle  \Big\Vert ^{2}}{\Big(   
  \int_{\Omega(\gamma_m)} \vert \det \gamma_m^{\sharp}(u)\vert dS_{u}  \Big)^{4/n}}
\end{equation*}
and on account of (2.10), we get
\begin{equation*}
\lim _{m\rightarrow \infty} \dfrac{\big\Vert  \beta_m \bullet \beta_m -k \langle.,
.\rangle \bullet \langle.,.\rangle  \big\Vert ^{2}}{\Big(   
  \int_{\Omega_k(\beta_m)} \vert \det \beta_m^{\sharp}(u)\vert dS_{u}  \Big)^{4/n}}=0.
\end{equation*}
Furthermore,  since $\|\gamma_m\|^2=1$ and $\lim _{m\rightarrow \infty}{\mathrm {sc}}(\gamma_m)=0$, we obtain
$$\lim _{m\rightarrow \infty}\dfrac{|{\mathrm {sc}}(\beta_m)|}{\|\beta_m\|^2}=0.$$

From these we conclude that there exist no positive constant depending only on $n,p,k$ such that  the inequality in Proposition 2.3 
 hold without the condition on the scalar curvature.
\end{example}

\begin{remark}
We have not explicitly  computed the constants that appear in Propositions  2.3 and 2.4.
 However these constants have the following property:
$$\liminf_{\delta\rightarrow 0^{+}} c(n,p,k,\delta)=0\;\;\text{and}\;\;\liminf_{\delta\rightarrow 0^{+}} c(n,p,\delta)=0.$$
We can easily get upper bounds for them. 
For instance, applying the inequality in Proposition 2.4 to the bilinear form $\beta$ in
 ${\mathrm {Hom}}(\mathbb{R}^n\times \mathbb{R}^n, \mathbb{R}^p)$
 given by 
$$\beta(x,y)=\big(\sum_{i=1}^{l}x_iy_i-\sum_{i=l+1}^{n}x_iy_i \big) \xi,$$ 
 where $1\leq l\leq n$,  $x=(x_1,\dots,x_n), y=(y_1,\dots,y_n)$  and $\xi$ is a unit vector in $W$, we find the following estimate

\begin{equation*}
c(n,p,\delta)\leq \dfrac{2^{4/n}\Big(8n^2(n-1)^2-\big((n-2l)^2-n  \big)^2\Big)
}{4n(n-1)\Big(\int_{S^{p-1}} |\langle u, \xi \rangle|^ndS_u\Big)^{4/n}}         
{\ }{\ } {\mbox{for}} {\ }{\ } \delta^{2}\leq \dfrac{(n-2l)^2-n}{n}.
\end{equation*}
\end{remark}

\section{Proofs}
At first we recall some basic and well known facts about total curvature and how Morse theory provides restrictions on the Betti numbers.

Let $f:(M^n,g)\longrightarrow\mathbb{R}^{n+p}$  be an isometric immersion of a compact, connected, oriented $n$-dimensional Riemannian
 manifold $(M^n,g)$ into the $(n+p)$-dimensional Euclidean space $\mathbb{R}^{n+p}$ equipped with the usual Riemannian metric
 $\langle .,. \rangle $. The normal bundle of $f$ is given by
\begin{equation*}
N_{f}=\lbrace(x,\xi)\in f^{*}(T\mathbb{R}^{n+p}):  \xi \perp df_{x}(T_{x}M) \rbrace,
\end{equation*}
where $f^{*}(T\mathbb{R}^{n+p})$ stands for the induced bundle, and the unit normal bundle of $f$ is defined by 
\begin{equation*}
UN_{f}=\lbrace(x,\xi)\in N_{f}:  |\xi | =1 \rbrace.
\end{equation*}

The generalized Gauss map $\nu:UN_{f}\longrightarrow {S}^{n+p-1}$ is given by $\nu(x,\xi)=\xi$, where ${S}^{n+p-1}$ is the unit 
$(n+p-1)$-sphere in $\mathbb{R}^{n+p}$. For each $u \in {S}^{n+p-1},$ we consider the height function $h_{u}:M^n\longrightarrow \mathbb{R}$ 
defined by $h_{u}(x)=\langle f(x),u\rangle, x\in M^n$. Since $h_{u}$ has a degenerate critical point if and only if $u$ is a 
critical value of the generalized Gauss map, by Sard's Theorem there exists a subset $E\subset{S}^{n+p-1}$ of zero measure such that $h_{u}$ 
is a Morse function for all $u \in {S}^{n+p-1} \setminus E$. For every $u \in {S}^{n+p-1} \setminus E$, we denote by $\mu_{i}(u)$ 
the number of critical points of $h_{u}$ of index $i$.  We also set $\mu_{i}(u)=0$ for any $u \in E$. According to Kuiper \cite{K}, 
the total curvature of index $i$ of $f$ is given by
\begin{equation*}
 \tau_{i}=\dfrac{1}{\mathrm{Vol}({S}^{n+p-1})}\int_{{S}^{n+p-1}}\mu_{i}(u)dS_{u},
\end{equation*}
where $dS$ denotes the volume element of the sphere ${S}^{n+p-1}$.

Let  $\beta_{i}=\dim H_{i}(M;\mathcal{F})$ be the $i$-th Betti number of $M$, 
where $H_{i}(M;\mathcal{F})$ is the $i$-th homology group with coefficients in a field $\mathcal{F}$. From the weak Morse inequalities \cite{M}, 
we know that $\mu_i(u)\geq \beta_{i}$ for every $u \in {S}^{n+p-1}$
such that $h_{u}$ is a Morse function. Integrating over ${S}^{n+p-1}$, we obtain 
\begin{equation}
\tau_i \geq \beta_{i}. 
\end{equation}

For each $(x,\xi) \in UN_{f},$ we denote by $A_{\xi}$ the shape operator of $f$ associated with the direction $\xi$ given by 
\begin{equation*}
g\big(A_{\xi}(X),Y\big) =   \langle\alpha(X,Y),\xi\rangle, 
\end{equation*}
where $X,Y$ are tangent to $M$ and $\alpha$ is the second fundamental form of $f$ viewed as a section of the vector bundle 
${\mathrm {Hom}}(TM\times TM, N_{f})$. There is a natural volume element $d\Sigma$ on the unit normal bundle $UN_{f}$. In fact, if $dV$ is 
a $(p-1)$-form on $UN_{f}$ such that its restriction to a fiber of the unit normal bundle at $(x,\xi)$ is the volume element of the 
unit $(p-1)$-sphere $S_x^{p-1}$ of the normal space of $f$ at $x,$ then $d\Sigma=dM\wedge dV$. Furthermore, we have 
\begin{equation*}
\nu^{*}(dS)=G(x,\xi)d\Sigma,
\end{equation*}
where $G(x,\xi):=(-1)^{n}\det A_{\xi}$ is the Lipschitz-Killing curvature at $(x,\xi) \in UN_{f}.$

A well-known formula due to Chern and Lashof \cite{CL2} states that 
\begin{equation}
 \int_{UN_{f}}|\det A_{\xi}|d\Sigma =  \sum_{i=0}^{n} \int_{{S}^{n+p-1}}\mu_{i}(u) dS_{u}.
\end{equation}

The total absolute curvature  $\tau (f)$ of $f$ in the sense of Chern and Lashof is defined by
\begin{equation*}
 \tau (f)=\dfrac{1}{\mathrm{Vol}({S}^{n+p-1})}\int_{UN_{f}}|\nu^{*}(dS)|=
\dfrac{1}{\mathrm{Vol}({S}^{n+p-1})}\int_{UN_{f}}|\det A_{\xi}|d\Sigma. 
\end{equation*} 

The following result is due to Chern and Lashof \cite{CL1,CL2}.

\begin{theorem}
Let $f:(M^n,g)\longrightarrow\mathbb{R}^{n+p}$ be an isometric immersion of a compact, connected, oriented, 
$n$-dimensional Riemannian manifold $(M^n,g)$ into $\mathbb{R}^{n+p}$. Then the total absolute curvature of $f$ satisfies the inequality
\begin{equation*}
\tau (f)\geq \sum_{i=0}^{n}\beta_{i}.
\end{equation*} 
\end{theorem}

For each $i \in \lbrace 0,...,n\rbrace$, we consider the subset $U^{i}N_{f}$ of the unit normal bundle of $f$ defined by 
\begin{equation*}
U^{i}N_{f}=\big\lbrace(x,\xi)\in UN_{f}: {\mathrm {Index}}(A_{\xi})=i  \big\rbrace.
\end{equation*}
 Shiohama and Xu \cite[Lemma p. 381]{SX} refined formula (3.2) as follows
\begin{equation}
 \int_{U^{i}N_{f}}|\det A_{\xi}|d\Sigma =  \int_{{S}^{n+p-1}} \mu_{i}(u) dS_{u}.
\end{equation}

We recall that the Riemannian curvature tensor of $(M^n,g)$ is the (0,4)-tensor $R$ that is related to the second fundamental form $\alpha$ 
via the Gauss equation
\begin{equation*}
R(X,Y,Z,W) = \langle \alpha(X,W), \alpha(Y,Z)\rangle-\langle \alpha(X,Z), \alpha(Y,W)\rangle, 
\end{equation*}
where $X,Y,Z,W$ are tangent vector fields of $M^n$.

By means of the Kulkarni-Nomizu product, the Gauss equation is written as
\begin{equation}
 R=-\dfrac{1}{2}\alpha \bullet \alpha.
\end{equation} 
Moreover, we consider the (0,4)-tensor $R_{1}$ given by
\begin{equation}
 R_{1}=-\dfrac{1}{2}g \bullet g.
\end{equation} 
\medskip 
We are now ready to give the proofs of the main results.

\begin{proof}[Proof of Theorem 1.1.]
Let $f:(M^n,g)\longrightarrow\mathbb{R}^{n+p}$ be an isometric immersion whose scalar curvature and 
 the second fundamental form satisfy $|\mathrm{scal}|\geq\lambda\Vert\alpha\Vert^2$. 
Appealing to Proposition 2.3, we have
\begin{equation*}
\big\Vert \alpha \bullet \alpha -kg \bullet g \big\Vert^{n/2}(x)\geq \big(c(n,p,k,\lambda)\big)^{n/4}\int_{\Omega_k (\alpha |_{x})} 
\vert \det A_{\xi}\vert dV_{\xi}
\end{equation*}
for all $x \in M^n$, where 
$$
\Omega_k (\alpha|_{x})=\left\{\begin{array}{l}\big\lbrace u \in S_x^{p-1}: p\leq {\mbox {Index}} ( \beta^{\sharp}(u) ) \leq n-p  \big \rbrace
\;\;\;\mbox{if}\;\;k>0,
\vspace*{1.5ex}\\
\;\;\;\;\;\;\;\;\;\;\;\;\;\;\;\;\;\;\;\;\;\;\;\;\;\;\;\;S_x^{p-1}\;\;\;\;\;\;\;\;\;\;\;\;\;\;\;\;\;\;\;\;\;\;\;\;\;\;\;\;\; \mbox{if}\;\;k\leq 0.
\end{array} \right.
$$
On account of  (3.4) and (3.5), the above inequality becomes 
\begin{equation*}
\big\Vert R -kR_{1} \big\Vert^{n/2}(x)\geq \big(\frac{1}{4}c(n,p,k,\lambda)\big)^{n/4}\int_{\Omega (\alpha|_{x})} \vert
 \det A_{\xi}\vert dV_{\xi}
\end{equation*}
for all $x \in M^n$. 
Integrating over $M^n$, we obtain
\begin{equation*}
\int _{M}\big\Vert R -kR_{1}\big\Vert^{n/2}dM \geq \big(\frac{1}{4}c(n,p,k,\lambda)\big)^{n/4} \sum_{i\in J} \int 
_{U^{i}N_{f}}\vert \det A_{\xi} \vert d\Sigma,
\end{equation*}
where
$$
J:=\left\{\begin{array}{l}\lbrace  p, \dots, n-p   \rbrace
\;\;\;\;\mbox{if}\;\;k>0,
\vspace*{1.5ex}\\
\lbrace  0, \dots, n  \rbrace\;\;\;\;\;\;\;\;\;\; \mbox{if}\;\;k\leq 0.
\end{array} \right.
$$
Bearing in mind (3.3), the definition of the total curvature of index $i$ and (3.1), 
we get
\begin{equation}
\int _{M}\big\Vert R -kR_{1} \big\Vert^{n/2}dM \geq \varepsilon(n,k,\lambda) 
\sum_{i\in J}\tau_{i}\geq\varepsilon(n,k,\lambda)\sum_{i\in J}\beta_{i},
\end{equation}
where the constant $\varepsilon(n,k,\lambda)$ is given by 
$$\varepsilon(n,k,\lambda):
=\min_{2\leq p \leq n/2} \big(\frac{1}{4}c(n,p,k,\lambda)\big)^{n/4}\mathrm{Vol}({S}^{n+p-1}).$$ 

Now suppose that $k>0$ and $\int _{M}\Vert R-kR_{1}\Vert^{n/2}dM <  \varepsilon(n,k,\lambda)$.
Then, in view of (3.6), we conclude that $\sum_{i=p}^{n-p} \tau_{i}<1$. Thus, there exists $u \in S^{n+p-1}$ 
such that the height function $h_{u}:M^n\longrightarrow \mathbb{R}$ is a Morse function whose number of critical points 
of index $i$ satisfies $\mu_{i}(u)=0$ for any $p\leq i\leq n-p$.  Appealing to the fundamental theorem of Morse theory
 (cf. \cite[Th. 3.5]{M} or \cite[Th. 4.10]{CE}), we deduce that $M^n$ has the homotopy type of a CW-complex 
with no cells of dimension $i$ for $p\leq i\leq n-p$.

For $k \leq 0$, from (3.6) it follows that every $M^n \in \mathcal{M}(n,k,\lambda,a)$ admits an isometric immersion into 
$\mathbb{R}^{n+p}, 2\leq p \leq n/2,$ and a height function $h$ 
with at most 
$a/\varepsilon(n,k,\lambda)$ critical points. Let $b_1< \dots <b_r$ be the critical values of $h$. From Morse Theory we know that
for any $b_i<t_1<t_2<b_{i+1}$, 
every connected component of $h^{-1}([t_1,t_2])$ is homeomorphic to $S^{n-1} \times [t_1,t_2]$. Hence there exists a number $s$ depending only on 
$n,k,\lambda,a$ such that the number of $n$-disks needed to cover $M^n$ is at most $s$.
\end{proof}

\begin{proof}[Proof of Theorem 1.2.]
Let $f:(M^n,g)\longrightarrow\mathbb{R}^{n+p}$ be an isometric immersion whose scalar curvature and 
 the second fundamental form satisfy $|\mathrm{scal}|\geq\lambda\Vert\alpha\Vert^2$. 
Appealing to Proposition 2.4, we have
\begin{equation*}
\Big\Vert \alpha \bullet \alpha -\dfrac{{\mathrm{scal}}}{n(n-1)}g \bullet g \Big\Vert^{n/2}(x)\geq
 \big(c(n,p,\lambda/)\big)^{n/4}\int_{\Omega(\alpha |_{x})} \vert \det A_{\xi}\vert dV_{\xi}
\end{equation*}
for all $x \in M^n$. 
By virtue of (3.4) and (3.5), the above inequality becomes 
\begin{equation*}
\Big\Vert R -\dfrac{{\mathrm{scal}}}{n(n-1)}R_{1} \Big\Vert^{n/2}(x)\geq
 \big(\frac{1}{4}c(n,p,\lambda)\big)^{n/4}\int_{\Omega(\alpha |_{x})} \vert \det A_{\xi}\vert dV_{\xi}
\end{equation*}
for all $x \in M^n$. 
Integrating over $M^n$, we obtain
\begin{equation*}
\int _{M}\Big\Vert R -\dfrac{{\mathrm{scal}}}{n(n-1)}R_{1}\Big\Vert^{n/2}dM \geq
 \big(\frac{1}{4}c(n,p,\lambda)\big)^{n/4} \sum_{i=p}^{n-p} \int _{U^iN_{f}}\vert \det A_{\xi} \vert d\Sigma.
\end{equation*}
Then the rest of the proof is the same as in Theorem 1.1.

If the scalar curvature of $M^n$ is non-positive, then  Proposition 2.4 yields  
\begin{equation*}
\Big\Vert \alpha \bullet \alpha -\dfrac{{\mathrm{scal}}}{n(n-1)}g \bullet g \Big\Vert^{n/2}(x)\geq
 \big(c(n,p,\lambda)\big)^{n/4}\int_{S_{x}^{p-1}} \vert \det A_{\xi}\vert dV_{\xi}
\end{equation*}
for all $x \in M^n$.  In view  of (3.4) and (3.5), the above inequality becomes 
\begin{equation*}
\Big\Vert R -\dfrac{{\mathrm{scal}}}{n(n-1)}R_{1} \Big\Vert^{n/2}(x)\geq
 \big(\frac{1}{4}c(n,p,\lambda)\big)^{n/4}\int_{S_{x}^{p-1}} \vert \det A_{\xi}\vert dV_{\xi}
\end{equation*}
for all $x \in M^n$. 
Integrating over $M^n$, we obtain
\begin{equation*}
\int _{M}\Big\Vert R -\dfrac{{\mathrm{scal}}}{n(n-1)}R_{1}\Big\Vert^{n/2}dM \geq
 \big(\frac{1}{4}c(n,p,\lambda)\big)^{n/4}  \int _{UN_{f}}\vert \det A_{\xi} \vert d\Sigma.
\end{equation*}
Bearing in mind the definition of the total absolute curvature $\tau (f)$ of $f$, 
we finally get
\begin{equation}
\int _{M}\Big\Vert R -\dfrac{{\mathrm{scal}}}{n(n-1)}R_{1} \Big\Vert^{n/2}dM \geq \varepsilon(n,\lambda) \tau (f),
\end{equation}
where  
$$\varepsilon(n,\lambda):=\min_{2\leq p \leq n/2}\big(\frac{1}{4}c(n,p,\lambda)\big)^{n/4}\mathrm{Vol}({S}^{n+p-1}).$$ 
Thus, the desired inequality follows immediately from (3.7) and Theorem 3.1.

The proof that the class $\mathcal{M}(n,\lambda,a)$ contains at most finitely many homeomorphism types is the same as in Theorem 1.1.
\end{proof}

\begin{proof}[Proof of Corollary 1.1.]
 We shall prove that $M^n$ is simply connected. Assume to the contrary that $M^n$ is not simply connected. 
Since the fundamental form $\pi_{1}(M^n)$ is finite, it contains a subgroup isomorphic to $\mathbb{Z}_{q}$ for 
some prime $q$. Let $\pi:\widetilde{M}^n\longrightarrow M^n$ be the Riemannian covering of $M^n$ corresponding to $\mathbb{Z}_{q}$. 
 Then $\widetilde{M}^n$ is compact, and we may appeal to Theorem 1.2 for the isometric immersion $f\circ\pi $, 
to conclude that  $H_{p}(\widetilde{M}^n;{\mathbb{Z}_{q}})=0$. 
 According to a deep result due
 to B${\ddot {\mbox {o}}}$hm and Wilking \cite{BW}, the universal covering of $\widetilde{M}^n$ is diffeomorphic to the sphere $S^{n}.$ 
Then, from the spectral sequence of the covering $S^{n}\longrightarrow \widetilde{M}^n$, we deduce that 
\begin{equation*}
H_{p}(\widetilde{M}^n;{\mathbb{Z}_{q}})\simeq H_{p}\big(K({\mathbb{Z}_{q}},1);{\mathbb{Z}_{q}}\big)={\mathbb{Z}_{q}},
\end{equation*}
contradiction. Thus, $M^n$ is simply connected and according to \cite{BW} it is diffeomorphic to the sphere $S^{n}.$
\end{proof}

\bibliographystyle{amsplain}

\begin{thebibliography}{18}

\bibitem {BW} C. B${\ddot {\mbox {o}}}$hm and B. Wilking, \textit{Manifolds with positive curvature operators are space forms},  
Ann. of Math. (2) \textbf{167} (2008), 1079--1097.

\bibitem{CaDa} M. do Carmo and M. Dajczer, \textit{Conformal rigidity}, Amer. J. Math.
\textbf{109} (1987),  963--985.

\bibitem {Ca} \'{E}. Cartan, \textit{Sur les vari\'{e}t\'{e}s de courbure constante d'un espace euclidien ou non-euclidien},  
Bull. Soc. Math. France \textbf{47} (1919), 125--160.

\bibitem {CE} J. Cheeger and D. Ebin, Comparison theorems in Riemannian geometry. North-Holland Mathematical Library, Vol. 9. 
North-Holland Publishing Co., Amsterdam-Oxford; American Elsevier Publishing Co., Inc., New York, 1975. 

\bibitem {CK} S.S. Chern and N.H. Kuiper, \textit{Some theorems on the isometric imbedding of compact Riemann manifolds  
in Euclidean space},  Ann. of Math. (2) \textbf{56} (1952), 422--430.

\bibitem {CL1} S.S. Chern and R.K. Lashof, \textit{On the total curvature of immersed manifolds}, Amer. J. Math. 
\textbf{79} (1957), 306--318.

\bibitem {CL2} S.S. Chern and R.K. Lashof, \textit{On the total curvature of immersed manifolds II}, Michigan Math. J. 
\textbf{5} (1958), 5--12.

\bibitem {Da} M. Dajczer, Submanifolds and isometric immersions. Based on the notes prepared by Mauricio Antonucci, Gilvan Oliveira, 
Paulo Lima-Filho and Rui Tojeiro. Mathematics Lecture Series, 13. Publish or Perish, Inc., Houston, TX, 1990.

\bibitem{Da1} M. Dajczer and L. Rodriguez, \textit{Rigidity of real Kaehler submanifolds}, Duke Math. J. 
\textbf{53} (1986),  211--220.

\bibitem{Da2} M. Dajczer and L.A. Florit, \textit{Compositions of isometric immersions in higher codimension},
manuscr. math. \textbf{105} (2001), 507--517.

\bibitem {K} N.H. Kuiper, \textit{Minimal total absolute curvature for immersions}, Invent. Math. 
\textbf{10} (1970), 209--238.

\bibitem {M} J. Milnor, Morse theory. Based on lecture notes by M. Spivak and R. Wells. Annals of Mathematics Studies, No. 51 
Princeton University Press, Princeton, N.J. 1963.

\bibitem{Mo1} J.D. Moore, \textit{Codimension two submanifolds of positive curvature},  Proc. Amer. Math. Soc.  \textbf{70}  (1978), 72--74. 

\bibitem{Mo2} J.D. Moore, \textit{Submanifolds of constant positive curvature. I.},  Duke Math. J. \textbf{44} (1977), 449--484.

\bibitem{Mo3} J.D. Moore, \textit{Conformally flat submanifolds of Euclidean space},  Math. Ann.  \textbf{225}  (1977), 89--97.

\bibitem{Mo4} J.D. Moore, \textit{Euler characters and submanifolds of constant positive curvature},  Trans. Amer. Math. Soc.  
\textbf{354} (2002), 3815--3834. 

\bibitem {O} T. Otsuki, \textit{Isometric embedding of Riemann manifolds in a Riemann manifold},  J. Math. Soc. Japan \textbf{6} (1954), 221--234.

\bibitem {SX} K. Shiohama and H. Xu, \textit{Lower bound for $L^{n/2}$ curvature norm and its application}, J. Geom. Anal.
\textbf{7} (1997), 377--386.




\end{thebibliography}

\end{document}